\documentclass[11pt]{article}

\usepackage[T1]{fontenc}
\usepackage[utf8]{inputenc}
\usepackage{lmodern}
\usepackage{amsmath,amssymb,amsthm,mathtools}
\usepackage[margin=1in]{geometry}
\usepackage{bm}
\usepackage{microtype}
\usepackage{booktabs,tabularx,array}
\usepackage{enumitem}
\usepackage{algorithm2e}
\usepackage{needspace,etoolbox}
\usepackage{xcolor}
\usepackage[numbers,sort&compress]{natbib}
\usepackage[hidelinks]{hyperref}
\usepackage{bookmark}

\IfFormatAtLeastTF{2026-06-01}{}{\usepackage{aliascnt}}
\usepackage[nameinlink,capitalize,noabbrev]{cleveref}

\newcommand{\ResearchAgentSystem}{our laboratory's internal auto-research system}

\newif\ifanonymous
\anonymousfalse

\newcommand{\AuthorOne}{Haihan Zhang}
\newcommand{\AuthorTwo}{Wendao Wu}
\newcommand{\AuthorThree}{Chenheng Zhang}
\newcommand{\AuthorFour}{Haoxuan Li}
\newcommand{\AuthorFive}{Zhouchen Lin}
\newcommand{\AuthorSix}{Cong Fang}
\newcommand{\AuthorSeven}{Yanyi Li}
\newcommand{\AuthorEight}{Chunyuan Zheng}
\newcommand{\AffiliationOne}{Peking University}

\newcommand{\EmailOne}{colombo.2026@outlook.com}

\newcommand{\EmailFour}{hxli@pku.edu.cn}
\newcommand{\EmailFive}{ZLIN@pku.edu.cn}

\ifanonymous
  \newcommand{\PDFAuthors}{Anonymous Authors}
\else
  \newcommand{\PDFAuthors}{\AuthorOne; \AuthorTwo; \AuthorThree; \AuthorFour; \AuthorFive}
\fi

\hypersetup{
  pdftitle={Near-Optimal Exact-Value Zeroth-Order Complexity for Smooth Strongly Convex Optimization},
  pdfauthor={\PDFAuthors},
  pdfsubject={Exact-value zeroth-order oracle complexity, smooth strongly convex optimization, finite differences, accelerated restart, and fixed-objective lower bounds},
  pdfkeywords={zeroth-order optimization, exact function values, smooth strongly convex optimization, oracle complexity, condition number, upper bound, lower bound}
}

\allowdisplaybreaks[1]

\setlist{leftmargin=*,itemsep=2pt,topsep=3pt}
\numberwithin{equation}{section}

\newcommand{\R}{\mathbb R}
\newcommand{\E}{\mathbb E}
\newcommand{\Prob}{\mathbb P}

\newcommand{\ip}[2]{\langle #1,#2\rangle}
\newcommand{\norm}[1]{\lVert #1\rVert}
\newcommand{\dist}{\operatorname{dist}}
\newcommand{\Lip}{\operatorname{Lip}}

\newcommand{\argminop}{\operatorname*{arg\,min}}

\newcommand{\Span}{\operatorname{span}}
\newcommand{\cF}{\mathcal F}
\newcommand{\cG}{\mathcal G}

\newcommand{\cC}{\mathcal C}

\newcommand{\1}{\bm 1}

\newcommand{\logp}[1]{\log_{+}\!\left(#1\right)}
\newcolumntype{P}[1]{>{\raggedright\arraybackslash}p{#1}}
\newcolumntype{Y}{>{\raggedright\arraybackslash}X}

\newtheorem{theorem}{Theorem}[section]

\IfFormatAtLeastTF{2026-06-01}{%
  \newtheorem{lemma}[theorem]{Lemma}
  \newtheorem{proposition}[theorem]{Proposition}
  \newtheorem{corollary}[theorem]{Corollary}
  \newtheorem{claim}[theorem]{Claim}
  \theoremstyle{definition}
  \newtheorem{definition}[theorem]{Definition}
  \newtheorem{remark}[theorem]{Remark}
  \newtheorem{example}[theorem]{Example}
}{%
  \newaliascnt{lemma}{theorem}
  \newtheorem{lemma}[lemma]{Lemma}
  \aliascntresetthe{lemma}
  \newaliascnt{proposition}{theorem}
  \newtheorem{proposition}[proposition]{Proposition}
  \aliascntresetthe{proposition}
  \newaliascnt{corollary}{theorem}
  \newtheorem{corollary}[corollary]{Corollary}
  \aliascntresetthe{corollary}
  \newaliascnt{claim}{theorem}
  
  \aliascntresetthe{claim}
  \theoremstyle{definition}
  \newaliascnt{definition}{theorem}
  \newtheorem{definition}[definition]{Definition}
  \aliascntresetthe{definition}
  \newaliascnt{remark}{theorem}
  \newtheorem{remark}[remark]{Remark}
  \aliascntresetthe{remark}
  \newaliascnt{example}{theorem}
  
  \aliascntresetthe{example}
}

\AtBeginEnvironment{theorem}{\Needspace{5\baselineskip}}
\AtBeginEnvironment{lemma}{\Needspace{5\baselineskip}}
\AtBeginEnvironment{proposition}{\Needspace{5\baselineskip}}
\AtBeginEnvironment{corollary}{\Needspace{5\baselineskip}}
\AtBeginEnvironment{claim}{\Needspace{5\baselineskip}}
\AtBeginEnvironment{definition}{\Needspace{5\baselineskip}}
\AtBeginEnvironment{remark}{\Needspace{4\baselineskip}}
\AtBeginEnvironment{example}{\Needspace{4\baselineskip}}
\AtBeginEnvironment{proof}{\Needspace{3\baselineskip}}

\title{\textbf{Near-Optimal Exact-Value Zeroth-Order Complexity for\\
Smooth Strongly Convex Optimization}}

\ifanonymous
  \author{Anonymous Authors}
\else
  \author{%
    \AuthorTwo$^{1,*}$ \quad
    \AuthorOne$^{1,*}$ \quad
    \AuthorThree$^{1,*}$\\[0.20em]
    \AuthorSeven$^{1}$ \quad
    \AuthorEight$^{1}$\\[0.20em]
    \AuthorSix$^{1,\dagger}$ \quad
    \AuthorFour$^{1,\dagger}$ \quad
    \AuthorFive$^{1,\dagger}$\\[0.60em]
    \small $^{1}$\AffiliationOne\\[0.35em]
    \small
    \href{mailto:\EmailOne}{\texttt{\EmailOne}}  \\[0.35em]
    \small $^{*}$Equal contribution.    \small $^{\dagger}$Corresponding authors.
  }
\fi
\date{}

\begin{document}
\maketitle

\begin{abstract}
We study deterministic adaptive optimization of globally
$\beta$-smooth, $\mu$-strongly convex functions using exact scalar
function values. Queries and outputs lie in $B_2^d(R)$, and the
minimizer lies in $B_2^d(R/2)$.
Set $\kappa=\beta/\mu$, $Q=\beta R^2/\epsilon$, and
$D_d=(d/\log(ed))^{1/3}$.
For sufficiently large $d$ and
$0<\epsilon\le c_\epsilon\beta R^2$, the minimax value complexity
$N_\epsilon$ satisfies
\[
\begin{aligned}
N_\epsilon
&\ge c d\min\{\sqrt Q,\sqrt\kappa,D_d\},\\
N_\epsilon
&\le C d\min\left\{
\sqrt Q,\sqrt\kappa[1+\log_+(Q/\kappa)]
\right\},
\end{aligned}
\]
where $c,C,c_\epsilon>0$ are universal constants and
$\log_+(t)=\max\{0,\log t\}$.
The lower bound uses an exactly shielded smooth chain and batched
delayed rotations; the upper bound combines finite differences,
acceleration, and restart.
When $\min\{Q,\kappa\}\le D_d^2$, these bounds match up to constants
in the accuracy-dominated regime $Q\le\kappa$ and at constant
relative accuracy $\epsilon=\Theta(\mu R^2)$.
For arbitrarily higher accuracy in the range $\kappa\le D_d^2$,
the bounds differ by at most $1+\log(\mu R^2/\epsilon)$;
the optimal accuracy dependence remains unresolved in general.
\end{abstract}


\noindent\textbf{AI Usage.}
Nearly the entire research pipeline for this paper was carried out by
\ResearchAgentSystem{}, powered by GPT-5.6 Sol. The system also conducted a
Lean-backed article audit of the resulting manuscript. The authors subsequently reviewed and approved the
mathematical claims, presentation, and formal artifacts, and take
responsibility for the final manuscript.The complete Lean audit report and the system's technical report will be made public at a later date.

\paragraph{Publication Intent and Recognition of Independent Work.} 
This preprint is released as a research output of our study of
automated scientific research, for public scrutiny and discussion.
As we do not intend to submit this manuscript to a journal or
conference, it has not undergone the full editorial refinement
and internal review we would normally require for submission.
The authors remain responsible for its contents and for correcting
any errors or inaccurate attribution.

We recognize that releasing these results may affect researchers
who have independently been pursuing the same questions, including
work that is not yet public.
If this applies to your work, we invite you to contact us at
\texttt{colombo.2026@outlook.com}.
A brief description of the overlap, together with any dated
materials you are comfortable sharing, would help us understand
its scope and chronology.
Such materials might include draft excerpts, presentation slides,
or revision records; a complete manuscript or unpublished proof
is not required to begin the conversation.

We will review these cases carefully and update citations,
acknowledgments, and accounts of independent or concurrent work
as appropriate, including revising priority claims where warranted.
Nonpublic materials will be treated confidentially and used only
to assess the overlap unless their authors authorize otherwise.
We will obtain explicit permission from the authors before
publicly discussing such materials or providing them to an
automated research system.

\section{Introduction}
\label{sec:introduction}
Zeroth-order optimization asks an algorithm to minimize an unknown
objective using only values of the function it queries.  In this paper the
feedback is exact: a query at $x$ returns the real number $f(x)$, and no
stochastic noise or finite-precision truncation is imposed.  We focus on
the smooth strongly convex case.  The objective belongs to
$\cF_{\mu,\beta}(d,R)$, meaning that it is globally $\beta$-smooth and
$\mu$-strongly convex, the condition number is
\[
  \kappa=\frac{\beta}{\mu},
\]
and the minimizer is promised to lie in $B_2^d(R/2)$ while all queries and
the output must remain in the public ball $B_2^d(R)$.  The goal is to
return $\widehat x$ with
\[
  f(\widehat x)-f^\star\le\epsilon.
\]
The normalized accuracy scale is
\[
  Q=\frac{\beta R^2}{\epsilon}.
\]
Thus $Q$ measures the smooth-convex difficulty, while $\kappa$ measures
the acceleration scale created by strong convexity.

At first sight, exact values make the upper-bound side almost formal.
With $d+1$ carefully chosen coordinate evaluations, one can approximate a
full gradient to arbitrarily high deterministic accuracy.  Substituting
this approximation into accelerated first-order methods suggests the value
complexity
\[
  d\sqrt\kappa
\]
per constant-factor reduction in the error, and the smooth-convex
alternative $d\sqrt Q$ when the requested accuracy is not yet in the
strongly convex regime.  This heuristic is correct, but turning it into a
bounded-query theorem requires more than citing a first-order method: every
finite-difference point must remain feasible, the gradient error must be
tracked through acceleration and projection, and the restart schedule must
avoid an unnecessary $\log\kappa$ burn-in loss.

\paragraph{The exact-value lower-bound frontier.}
The lower-bound side is substantially less automatic.  Since one exact
scalar value may contain infinitely many bits, finite-alphabet or
bit-counting arguments alone cannot establish the desired lower bound.
A resisting oracle must
also be globally consistent with one fixed objective after all adaptive
queries are chosen.  In nonsmooth exact-value optimization, recent work
has shown that exact scalar feedback can still be information-limited in
high dimension, but smooth strongly convex objectives require a different
construction: generic smoothing leaks hidden directions through exact
values, while strong convexity ties the terminal error to the location of
the minimizer.  The open question addressed here is therefore not merely
whether finite differences give a good algorithm, but whether the
accelerated value scale is unavoidable for deterministic adaptive methods
under the same bounded-query model.

\paragraph{Our answer.}
We identify the accelerated value complexity up to a single
relative-accuracy logarithm in a precisely specified high-dimensional
region.  Let
\[
  D_d=\left(\frac{d}{\log(ed)}\right)^{1/3}.
\]
For sufficiently large $d$ and $Q$, our main upper and lower bounds give
\begin{equation}
\boxed{
  c d\min\{\sqrt Q,\sqrt\kappa,D_d\}
  \le
  N_{\epsilon,\mathrm{det-ad}}^{\mathrm{sc,val}}(d,R,\mu,\beta)
  \le
  C d\min\left\{
    \sqrt Q,
    \sqrt\kappa\left[1+\logp{Q/\kappa}\right]
  \right\}.}
\label{eq:intro-main-sandwich}
\end{equation}
Consequently, if
\begin{equation}
  \min\{Q,\kappa\}\le D_d^2,
  \label{eq:intro-certified-region}
\end{equation}
then the only remaining gap is the relative-accuracy logarithm
$1+\logp{Q/\kappa}$.  Two important special cases follow immediately.
If $Q\le\kappa$ and $Q\le D_d^2$, the value complexity is exactly
\[
  \Theta(d\sqrt Q)
  \equiv
  \Theta\!\left(d\sqrt{\frac{\beta R^2}{\epsilon}}\right)
\]
up to universal constants.  If $\kappa\le Q$ and
$\kappa\le D_d^2$, then
\[
  N_{\epsilon,\mathrm{det-ad}}^{\mathrm{sc,val}}
  =\widetilde\Theta(d\sqrt\kappa),
\]
where the tilde hides only a logarithm in the relative accuracy
$\mu R^2/\epsilon$.  Thus the result extends the familiar
constant-relative-accuracy $d\sqrt\kappa$ law to arbitrary high accuracy
in this condition-number range.  At constant relative accuracy,
$\epsilon=\Theta(\mu R^2)$, the bounds reduce to
$\Theta(d\sqrt\kappa)$, with no remaining logarithmic gap.

\paragraph{Why the result is significant.}
The theorem clarifies the role of strong convexity in exact-value
zeroth-order optimization.  Strong convexity does not remove the leading
factor $d$ needed to recover directional information from scalar values,
but it replaces the smooth-convex accuracy scale $\sqrt Q$ by the
condition-number scale $\sqrt\kappa$ once the target accuracy enters the
strongly convex regime.  The lower bound shows that this transition is not
an artifact of coordinate finite differences or of a particular
accelerated method.  It is forced by a fixed smooth strongly convex
objective whose exact scalar values remain locally shielded from the
algorithm until enough directions have been learned.

\subsection{Contributions}
The paper makes five main contributions.
\begin{enumerate}[label=\textup{(C\arabic*)}]
\item \textbf{A deterministic bounded-query upper bound.}
We establish the upper bound in \eqref{eq:intro-main-sandwich} with a
self-contained analysis that counts every exact scalar value call.
Feasible finite differences, robust projected acceleration, a single
burn-in phase, and certified restarts yield the relative-accuracy logarithm
without an additional $\log\kappa$ term.

\item \textbf{A fixed-parameter smooth strongly convex lower bound.}
We prove the lower bound in \eqref{eq:intro-main-sandwich} against all
deterministic adaptive bounded-query algorithms.  For each algorithm, the
hard instance is a single fixed $C^{1,1}$ objective.  The calibration is
stated directly in the public parameters $(\mu,\beta,\epsilon)$, rather
than through an implicit chain length.

\item \textbf{An exact shielding mechanism for smooth value oracles.}
The lower bound uses a Moreau-smoothed biased max chain whose prefix is
exactly shielded.  This prevents exact function values from leaking
future hidden directions while preserving global $C^{1,1}$ smoothness.

\item \textbf{A delayed-rotation compiler with one fixed objective.}
Hidden directions are chosen in blocks after the algorithm's queries are
fixed, but the transcript is later realized by a single deterministic
objective.  A sentinel direction and a hidden-shift quadratic enforce the
interior minimizer promise and the exact strong-convexity modulus.

\item \textbf{A precise near-optimality landscape.}
We identify the certified region \eqref{eq:intro-certified-region}, the
exact accuracy-dominated and constant-relative-accuracy regimes, and the
high-accuracy $\widetilde\Theta(d\sqrt\kappa)$ branch.  The remaining
logarithmic and dimensional gaps are kept explicit.
\end{enumerate}

\subsection{Technical overview}
The upper bound starts from the elementary fact that exact coordinate
finite differences can simulate a gradient with arbitrary deterministic
accuracy.  The main issue is making this simulation compatible with the
bounded domain.  We therefore use forward or one-sided differences whose
query points stay in $B_2^d(R)$ and prove an inexact first-order oracle
inequality.  This feeds into an accelerated projected method robust to
summable deterministic gradient errors.  Running the method without strong
convexity gives the direct branch $O(d\sqrt Q)$.

For the strongly convex branch, one accelerated phase first enters a
region where the function gap is $O(\mu R^2)$.  Restarted accelerated
phases then reduce the certified error geometrically.  Each phase costs
$O(d\sqrt\kappa)$ value calls, and the number of phases is
$O(1+\log(\mu R^2/\epsilon))$.  Since the burn-in and restart analysis are
performed in the same value-oracle model, this yields the upper bound in
\cref{thm:upper-main} without an extra $\log\kappa$ term.

The lower bound is organized around a chain length $m$.  A biased max
chain creates a terminal gap of order $\beta_0R^2/m^2$, while its Moreau
envelope gives the required smoothness.  The key identity is exact prefix
shielding: within a shielded region, values depend only on the already
revealed prefix of the hidden orthogonal frame.  Batched delayed rotations
then force roughly $d$ scalar queries for each new chain coordinate, so
fewer than $\Omega(dm)$ calls cannot reveal the final direction.  The
compiler produces a fixed function with strong-convexity modulus
$\beta_0/(4096m^2)$.  Choosing $m$ as the largest admissible integer below
\[
  \min\{\sqrt Q,\sqrt\kappa,D_d\}
\]
calibrates the construction to the desired public parameters and proves
\cref{thm:lower-main}.

\subsection{Scope and organization}
The results are information-complexity statements.  Exact arithmetic and
internal computation are free, but every scalar value query is charged and
must lie in the public ball.  The lower bound covers fully adaptive
deterministic algorithms, but not randomized algorithms or algorithms
allowed to query outside $B_2^d(R)$.  A matching relative-accuracy
logarithmic lower bound and near-optimality beyond
$\min\{Q,\kappa\}\le D_d^2$ remain open.  The hard objectives are globally
$C^{1,1}$; higher smoothness remains outside the present proof.

\Cref{sec:related} compares the oracle conventions and rate statements in
the relevant literature.  \Cref{sec:model} states the formal model and the
main theorems.  \Cref{sec:upper} proves the finite-difference upper
bounds.  \Cref{sec:lower-architecture}--\cref{sec:admissibility} build
the shielded value chain and delayed-rotation compiler.  \Cref{sec:calibration}
calibrates the hard instance to $(\mu,\beta,\epsilon)$, and
\cref{sec:landscape} summarizes the resulting complexity landscape.

\section{Related work and rate comparison}
\label{sec:related}
\paragraph{Oracle convention.}
The distinction between exact deterministic values, noisy values, and
first-order gradients is central.  In the present model, a query returns a
single real number $f(x)$ with no stochastic error and no finite-precision
truncation.  The algorithm may adapt to the entire transcript, but every
query must remain in the bounded feasible ball.  This convention is the
one used simultaneously in the upper-bound simulation and in the
fixed-objective lower bound.

\paragraph{Classical value-oracle complexity.}
Information-based complexity for convex optimization from function values
goes back to the deterministic oracle framework of \citet{nemirovski1983problem}
and related value-only methods such as \citet{protasov1996algorithms}.
Modern zeroth-order optimization developed sharp stochastic and randomized
rates for convex classes; see, for example,
\citet{shamir2013complexity,duchi2015optimal,shamir2017optimal,nesterov2017random}.
For smooth convex objectives, accelerated randomized gradient-free methods
achieve the value scale corresponding to $d\sqrt Q$ calls
\citep{nesterov2017random}.  Our direct upper proof is deterministic and
keeps all finite-difference evaluations inside the same bounded query set
used by the lower bound.

\paragraph{Strong convexity and finite-difference acceleration.}
The $\sqrt\kappa$ iteration scale for smooth strongly convex first-order
optimization is classical \citep{nesterov2004introductory,nesterov2018lectures}.
Because exact values allow arbitrarily accurate coordinate finite
differences, this first-order theory strongly suggests the upper rate
$d\sqrt\kappa$ per constant-factor error reduction.  Robust accelerated
methods under additive gradient errors have been analyzed in first-order
and derivative-free settings \citep{vasin2023accelerated,gorbunov2022accelerated}.
The proof here isolates the deterministic bounded-query version needed
for the exact-value complexity theorem.

\paragraph{Lower bounds from fixed objectives.}
Resisting rotations and chain constructions are standard tools for oracle
lower bounds \citep{carmon2020lower,diakonikolas2019parallel}.  Exact
scalar feedback introduces an additional consistency constraint: all
answers must be values of one fixed function, not merely responses from an
adaptive adversary.  Recent exact-value lower bounds for nonsmooth
Lipschitz classes obtain near-quadratic dimension dependence
\citep{kerger2026closing,zhang2026nearoptimal}.  The present construction
addresses the smooth strongly convex regime by combining a Moreau-smoothed
biased max chain with exact prefix shielding.

To compare the rates, set
\[
  Q=\frac{\beta R^2}{\epsilon},
  \qquad
  \kappa=\frac{\beta}{\mu},
  \qquad
  D_d=\left(\frac{d}{\log(ed)}\right)^{1/3}.
\]
The table suppresses universal constants and, where indicated, logarithmic
factors.

\begin{table}[t]
\centering
\caption{Representative rates and scopes for value-oracle and related
smooth strongly convex optimization results.  Rows with different oracle
models are not directly comparable.}
\label{tab:related-rates}
\footnotesize
\setlength{\tabcolsep}{4.0pt}
\renewcommand{\arraystretch}{1.18}
\begin{tabularx}{\linewidth}{@{}P{1.95cm}P{1.90cm}P{3.10cm}Y P{1.55cm}@{}}
\toprule
Setting & Oracle / criterion & Representative complexity & Scope and limitation & Reference \\
\midrule
Smooth convex optimization
& exact or randomized values; function gap
& $O(d\sqrt Q)$ upper scale
& No strong-convexity acceleration branch; randomized analyses often allow stochastic or two-point estimators.
& \citep{nesterov2017random,duchi2015optimal,shamir2017optimal} \\
\addlinespace
Smooth strongly convex optimization
& first-order gradients; function gap
& $O(\sqrt\kappa\log(1/\epsilon))$ gradient calls
& Does not count the $d$ scalar evaluations needed to recover one gradient from exact values.
& \citep{nesterov2004introductory,nesterov2018lectures} \\
\addlinespace
Smooth strongly convex optimization
& deterministic exact values; upper bound
& $\begin{gathered}O(d\min\{\sqrt Q,\\\sqrt\kappa[1+\logp{Q/\kappa}]\})\end{gathered}$
& Feasible coordinate finite differences, robust acceleration, burn-in, and restart; every value query is charged.
& This work \\
\addlinespace
Smooth strongly convex optimization
& deterministic exact values; lower bound
& $\begin{gathered}\Omega(d\min\{\sqrt Q,\\\sqrt\kappa,D_d\})\end{gathered}$
& Fixed $C^{1,1}$ objective; deterministic adaptive algorithms; bounded queries; cubic dimension threshold.
& This work \\
\addlinespace
Nonsmooth exact-value convex optimization
& deterministic exact values; function gap
& near-quadratic dimension lower-bound phenomena
& Different regularity class; shows why exact scalar feedback still admits information-theoretic lower bounds.
& \citep{kerger2026closing,zhang2026nearoptimal} \\
\bottomrule
\end{tabularx}
\end{table}

The comparison highlights the contribution of the present theorem.  The
upper bound confirms that exact values can simulate accelerated gradients
with only the unavoidable leading factor $d$.  The lower bound shows that,
in the certified high-dimensional region, no deterministic adaptive
bounded-query algorithm can remove this $d\sqrt{\min\{Q,\kappa\}}$ core.
The remaining mismatch is sharply localized: a logarithm in relative
accuracy, the threshold $D_d$, and the deterministic bounded-query scope.

\section{Problem formulation and main results}
\label{sec:model}

Throughout, $d\in\mathbb N$, $R>0$, and $0<\mu\le\beta$.  Let
\[
X:=B_2^d(R),
\qquad
\cC:=B_2^d(R/2),
\qquad
\kappa:=\frac{\beta}{\mu}.
\]
We use the standard convention
\[
\log_+(u):=\max\{0,\log u\},
\qquad u>0.
\]

\begin{definition}[Auxiliary smooth convex class]
\label{def:convex-class}
For $L>0$, let $\cG_L(d,R)$ be the set of all functions $f:\R^d\to\R$ such that
\begin{enumerate}[label=(\roman*)]
\item $f$ is convex and continuously differentiable;
\item $\nabla f$ is globally $L$-Lipschitz in Euclidean norm;
\item $f$ has a unique global minimizer $x_f^\star\in\cC$.
\end{enumerate}
\end{definition}

\begin{definition}[Smooth strongly convex class]
\label{def:sc-class}
Let $\cF_{\mu,\beta}(d,R)$ be the set of all continuously differentiable functions $f:\R^d\to\R$ satisfying
\begin{align}
f(y)&\ge f(x)+\ip{\nabla f(x)}{y-x}+\frac\mu2\norm{y-x}_2^2,
\label{eq:sc}\\
\norm{\nabla f(x)-\nabla f(y)}_2&\le\beta\norm{x-y}_2
\label{eq:smooth}
\end{align}
for all $x,y\in\R^d$, and whose unique minimizer $x_f^\star$ lies in $\cC$.  Write $f^\star=f(x_f^\star)$.
\end{definition}

A deterministic adaptive algorithm with budget $T$ consists of Borel maps
\[
Q_t:(X\times\R)^{t-1}\to X,
\qquad t=1,\ldots,T,
\]
and a Borel terminal map
\[
Q_{\mathrm{out}}:(X\times\R)^T\to X.
\]
Recursively,
\[
x_t=Q_t((x_1,f(x_1)),\ldots,(x_{t-1},f(x_{t-1}))),
\]
and
\[
\widehat x=Q_{\mathrm{out}}((x_1,f(x_1)),\ldots,(x_T,f(x_T))).
\]
Repeated calls are charged unless a stored value is reused.  Internal exact-real computation is free, and the objective is fixed throughout the interaction.

\begin{definition}[Minimax exact-value complexity]
For $\epsilon>0$, define
\[
N_{\epsilon,\mathrm{det-ad}}^{\mathrm{sc,val}}(d,R,\mu,\beta)
:=
\inf\left\{
T\in\mathbb N_0:
\begin{array}{l}
\text{there exists a deterministic adaptive $T$-call algorithm}\\[1mm]
\text{such that }\displaystyle\sup_{f\in\cF_{\mu,\beta}(d,R)}
[f(\widehat x_f)-f^\star]\le\epsilon
\end{array}
\right\}.
\]
As usual, $\inf\varnothing=+\infty$.
\end{definition}

Set
\[
Q:=\frac{\beta R^2}{\epsilon},
\qquad
D_d:=\left(\frac{d}{\log(ed)}\right)^{1/3}.
\]

\begin{theorem}[Combined deterministic upper bound]
\label{thm:upper-main}
There is a universal constant $C>0$ such that, for every $d\in\mathbb N$, $R>0$, $0<\mu\le\beta$, and $0<\epsilon\le\beta R^2$,
\begin{equation}
\label{eq:upper-main}
N_{\epsilon,\mathrm{det-ad}}^{\mathrm{sc,val}}(d,R,\mu,\beta)
\le
C d\min\left\{
\sqrt Q,
\sqrt\kappa\left[1+\logp{Q/\kappa}\right]
\right\}.
\end{equation}
Every query made by the constructed algorithms lies in $X$, and their outputs lie in $\cC$.
\end{theorem}

The lower proof is organized around the following auxiliary compiler.

\begin{theorem}[Exactly shielded hard-instance compiler]
\label{thm:compiler}
There exist universal constants $c_0>0$ and $d_0\in\mathbb N$ such that the following holds.  Let $d\ge d_0$, $R,\beta_0>0$, and let $m$ be an integer satisfying
\begin{equation}
\label{eq:m-conditions}
1\le m\le\frac d4,
\qquad
m^3\log(ed)\le c_0d.
\end{equation}
For every deterministic adaptive algorithm making fewer than $dm/8$ exact-value calls in $X$, there exists one fixed function $F\in\cG_{\beta_0}(d,R)$ such that
\begin{equation}
\label{eq:integer-gap}
F(\widehat x)-F^\star
\ge
2^{-17}\frac{\beta_0R^2}{m^2}.
\end{equation}
Moreover, the largest strong-convexity modulus of $F$ is exactly
\begin{equation}
\label{eq:compiler-mu}
\mu(F)=\frac{\beta_0}{4096m^2}.
\end{equation}
\end{theorem}

\begin{theorem}[Fixed-parameter lower bound]
\label{thm:lower-main}
There exist universal constants $c,c_\epsilon>0$ and $d_0\in\mathbb N$ such that, whenever
\[
d\ge d_0,
\qquad
R>0,
\qquad
0<\mu\le\beta,
\qquad
0<\epsilon\le c_\epsilon\beta R^2,
\]
one has
\begin{equation}
\label{eq:lower-main}
N_{\epsilon,\mathrm{det-ad}}^{\mathrm{sc,val}}(d,R,\mu,\beta)
\ge
c d\min\left\{
\sqrt Q,
\sqrt\kappa,
D_d
\right\}.
\end{equation}
\end{theorem}

\begin{theorem}[Near-optimal region]
\label{thm:near-region}
Under the assumptions of \cref{thm:lower-main}, suppose additionally that
\begin{equation}
\label{eq:near-condition}
\min\{Q,\kappa\}
\le
\left(\frac{d}{\log(ed)}\right)^{2/3}.
\end{equation}
Then universal constants $c,C>0$ satisfy
\begin{equation}
\label{eq:near-main}
c d\sqrt{\min\{Q,\kappa\}}
\le
N_{\epsilon,\mathrm{det-ad}}^{\mathrm{sc,val}}(d,R,\mu,\beta)
\le
C d\sqrt{\min\{Q,\kappa\}}
\left[1+\logp{Q/\kappa}\right].
\end{equation}
\end{theorem}

\begin{corollary}[Exact accuracy-dominated branch]
\label{cor:accuracy-branch}
If
\[
Q\le\kappa,
\qquad
Q\le\left(\frac{d}{\log(ed)}\right)^{2/3},
\]
then, subject to the universal thresholds of \cref{thm:lower-main},
\[
N_{\epsilon,\mathrm{det-ad}}^{\mathrm{sc,val}}(d,R,\mu,\beta)
=
\Theta\left(d\sqrt{\frac{\beta R^2}{\epsilon}}\right).
\]
\end{corollary}

\begin{corollary}[Full high-accuracy $d\sqrt\kappa$ region]
\label{cor:kappa-branch}
There exist universal constants $c,C,c_\epsilon>0$ and $d_0\in\mathbb N$ such that, whenever
\[
d\ge d_0,
\qquad
\kappa\le\left(\frac{d}{\log(ed)}\right)^{2/3},
\qquad
0<\epsilon\le c_\epsilon\mu R^2,
\]
one has
\begin{equation}
\label{eq:kappa-branch}
c d\sqrt\kappa
\le
N_{\epsilon,\mathrm{det-ad}}^{\mathrm{sc,val}}(d,R,\mu,\beta)
\le
C d\sqrt\kappa
\left[1+\log\frac{\mu R^2}{\epsilon}\right].
\end{equation}
Consequently,
\[
N_{\epsilon,\mathrm{det-ad}}^{\mathrm{sc,val}}
=
\widetilde\Theta(d\sqrt\kappa)
\]
throughout this entire high-accuracy region, where the suppressed factor is logarithmic only in $\mu R^2/\epsilon$.
\end{corollary}

\begin{remark}[Exact versus near-optimal]
At constant relative accuracy, $\epsilon=\Theta(\mu R^2)$, \eqref{eq:kappa-branch} reduces to $\Theta(d\sqrt\kappa)$.  For arbitrarily smaller $\epsilon$, the current lower bound does not contain the logarithmic factor in \eqref{eq:kappa-branch}; the result is therefore near-optimal rather than an exact all-accuracy characterization.
\end{remark}

\begin{proposition}[The degenerate endpoint $\kappa=1$]
\label{prop:kappa-one}
If $\beta=\mu$, then every $f\in\cF_{\mu,\mu}(d,R)$ is a shifted isotropic quadratic.  Consequently, for every $\epsilon>0$,
\[
N_{\epsilon,\mathrm{det-ad}}^{\mathrm{sc,val}}(d,R,\mu,\mu)\le d+1.
\]
Moreover, for $0<\epsilon\le\mu R^2/64$ and $d\ge2$,
\[
d-1
\le
N_{\epsilon,\mathrm{det-ad}}^{\mathrm{sc,val}}(d,R,\mu,\mu)
\le d+1.
\]
\end{proposition}

\begin{proof}
Summing the two strong-convexity inequalities with $(x,y)$ interchanged gives
\[
\ip{\nabla f(x)-\nabla f(y)}{x-y}
\ge
\mu\norm{x-y}_2^2.
\]
On the other hand, smoothness and Cauchy--Schwarz give
\[
\ip{\nabla f(x)-\nabla f(y)}{x-y}
\le
\norm{\nabla f(x)-\nabla f(y)}_2\norm{x-y}_2
\le
\mu\norm{x-y}_2^2.
\]
Equality throughout implies
\[
\nabla f(x)-\nabla f(y)=\mu(x-y),
\]
so $f(x)=\frac\mu2\norm{x}_2^2+\ip b x+c$ for some $b\in\R^d$ and $c\in\R$.  Querying $f(0)$ and $f(he_i)$ for $i=1,\ldots,d$, with any fixed $0<h\le R$, recovers
\[
b_i=\frac{f(he_i)-f(0)}{h}-\frac{\mu h}{2}.
\]
The minimizer $-b/\mu$, promised to lie in $\cC$, can therefore be output exactly after $d+1$ calls.  The lower inequality follows from \cref{lem:quadratic} with $\lambda=\mu$.
\end{proof}

\section{Complete proof of the upper bounds}
\label{sec:upper}

All accelerated iterates are constrained to the inner ball
\[
\cC=B_2^d(R/2),
\]
which contains the promised minimizer and has diameter $R$.  A forward finite-difference displacement of length at most $R/2$ therefore remains in the allowed query ball $X=B_2^d(R)$.

\subsection{A deterministic finite-difference oracle}

\begin{lemma}[Forward differences]
\label{lem:fd}
Let $f\in\cG_\beta(d,R)$, $y\in\cC$, and $0<h\le R/2$.  Define
\begin{equation}
\label{eq:fd}
[g_h(y)]_i
:=
\frac{f(y+he_i)-f(y)}{h},
\qquad i=1,\ldots,d.
\end{equation}
Then all $d+1$ queried points belong to $X$, and
\begin{equation}
\label{eq:fd-error}
\norm{g_h(y)-\nabla f(y)}_2
\le\frac{\beta h\sqrt d}{2}.
\end{equation}
\end{lemma}

\begin{proof}
The triangle inequality gives
\[
\norm{y+he_i}_2\le\norm y_2+h\le R,
\]
so every query is admissible.  Smoothness along the line $t\mapsto y+te_i$ gives
\[
\left|f(y+he_i)-f(y)-h\,\partial_i f(y)\right|
\le\frac{\beta h^2}{2}.
\]
After division by $h$, each coordinate error is at most $\beta h/2$.  Summing the squared coordinate bounds proves \eqref{eq:fd-error}.
\end{proof}

\begin{lemma}[Inexact oracle inequalities]
\label{lem:inexact-oracle}
Let $x,y\in\cC$ and suppose
\[
\norm{g-\nabla f(y)}_2\le\delta_g.
\]
Set
\begin{equation}
\label{eq:oracle-errors}
\overline\beta:=2\beta,
\qquad
\delta_\ell:=R\delta_g,
\qquad
\delta_u:=\frac{\delta_g^2}{2\beta}.
\end{equation}
Then
\begin{align}
f(x)
&\ge f(y)+\ip{g}{x-y}-\delta_\ell,
\label{eq:oracle-lower}\\
f(x)
&\le f(y)+\ip{g}{x-y}
+\frac{\overline\beta}{2}\norm{x-y}_2^2+\delta_u.
\label{eq:oracle-upper}
\end{align}
\end{lemma}

\begin{proof}
Write $e=g-\nabla f(y)$.  Convexity gives
\[
f(x)\ge f(y)+\ip{\nabla f(y)}{x-y}
=f(y)+\ip g{x-y}-\ip e{x-y}.
\]
Since $\operatorname{diam}(\cC)=R$, the final inner product is at most $R\delta_g$ in absolute value, proving \eqref{eq:oracle-lower}.  By $\beta$-smoothness and Young's inequality,
\begin{align*}
f(x)
&\le f(y)+\ip{g}{x-y}
+\delta_g\norm{x-y}_2+\frac\beta2\norm{x-y}_2^2\\
&\le f(y)+\ip{g}{x-y}
+\beta\norm{x-y}_2^2+\frac{\delta_g^2}{2\beta},
\end{align*}
which is \eqref{eq:oracle-upper}.
\end{proof}

\subsection{An error-robust accelerated projected method}

Fix an arbitrary initial point $x_0=z_0\in\cC$, let $A_0=0$, and define
\[
\psi_0(x):=\frac12\norm{x-x_0}_2^2,
\qquad x\in\cC.
\]
For $k=0,\ldots,N-1$, choose the unique $\alpha_{k+1}>0$ satisfying
\begin{equation}
\label{eq:alpha}
A_{k+1}:=A_k+\alpha_{k+1}
=\overline\beta\alpha_{k+1}^2.
\end{equation}
Given $x_k,z_k\in\cC$, set
\begin{equation}
\label{eq:upper-y}
y_{k+1}
:=
\frac{A_kx_k+\alpha_{k+1}z_k}{A_{k+1}}.
\end{equation}
At $y_{k+1}$ obtain $f(y_{k+1})$ and an approximate gradient $g_{k+1}$ satisfying \cref{lem:inexact-oracle}.  Update
\begin{align}
\psi_{k+1}(x)
&:=
\psi_k(x)+\alpha_{k+1}
\left[f(y_{k+1})+\ip{g_{k+1}}{x-y_{k+1}}\right],
\label{eq:psi-update}\\
z_{k+1}
&:=\argminop_{x\in\cC}\psi_{k+1}(x),
\label{eq:z-update}\\
x_{k+1}
&:=
\frac{A_kx_k+\alpha_{k+1}z_{k+1}}{A_{k+1}}.
\label{eq:x-update}
\end{align}
All three points $y_{k+1},z_{k+1},x_{k+1}$ lie in $\cC$.  Moreover,
\[
z_{k+1}
=
\Pi_{\cC}\left(x_0-\sum_{i=1}^{k+1}\alpha_i g_i\right).
\]

\begin{lemma}[Accelerated estimate with deterministic errors]
\label{lem:accelerated-error}
Suppose the same nonnegative numbers $\delta_\ell,\delta_u$ satisfy \eqref{eq:oracle-lower}--\eqref{eq:oracle-upper} at every iteration.  If $x^\star\in\argminop_{x\in\cC}f(x)$, then
\begin{equation}
\label{eq:accelerated-bound}
f(x_N)-f(x^\star)
\le
\frac{\norm{x^\star-x_0}_2^2}{2A_N}
+(N+1)\delta_\ell+N\delta_u,
\end{equation}
and
\begin{equation}
\label{eq:A-lower}
A_N\ge\frac{N^2}{4\overline\beta}
=\frac{N^2}{8\beta}.
\end{equation}
\end{lemma}

\begin{proof}
Write $\psi_k^*:=\min_{x\in\cC}\psi_k(x)=\psi_k(z_k)$.  We prove by induction that
\begin{equation}
\label{eq:estimate-induction}
\psi_k^*\ge A_k f(x_k)-E_k,
\end{equation}
where $E_0=0$ and
\begin{equation}
\label{eq:E-recursion}
E_{k+1}:=E_k+A_k\delta_\ell+A_{k+1}\delta_u.
\end{equation}
The claim is trivial at $k=0$.  Since $\psi_k$ is $1$-strongly convex and $z_k$ minimizes it over the closed convex set $\cC$, first-order optimality gives
\begin{equation}
\label{eq:psi-strong}
\psi_k(z_{k+1})
\ge
\psi_k(z_k)+\frac12\norm{z_{k+1}-z_k}_2^2.
\end{equation}
Let $\Delta z_k:=z_{k+1}-z_k$, $\alpha:=\alpha_{k+1}$, $A:=A_{k+1}$, and $y:=y_{k+1}$.  Combining \eqref{eq:psi-update}, \eqref{eq:psi-strong}, and the induction hypothesis yields
\[
\psi_{k+1}^*
\ge
A_k f(x_k)-E_k
+\frac12\norm{\Delta z_k}_2^2
+\alpha\left[f(y)+\ip{g_{k+1}}{z_{k+1}-y}\right].
\]
The lower oracle inequality at $x_k$ implies
\[
A_k f(x_k)
\ge
A_k\left[f(y)+\ip{g_{k+1}}{x_k-y}-\delta_\ell\right].
\]
Using
\begin{equation}
\label{eq:coupling-identity}
A_k(x_k-y)+\alpha(z_{k+1}-y)=\alpha\Delta z_k,
\end{equation}
we obtain
\begin{equation}
\label{eq:estimate-mid}
\psi_{k+1}^*
\ge
A f(y)+\alpha\ip{g_{k+1}}{\Delta z_k}
+\frac12\norm{\Delta z_k}_2^2
-E_k-A_k\delta_\ell.
\end{equation}
On the other hand,
\[
x_{k+1}-y=\frac\alpha A\Delta z_k.
\]
Applying the upper oracle inequality and using $A=\overline\beta\alpha^2$ gives
\begin{align*}
A f(x_{k+1})
&\le
A f(y)+\alpha\ip{g_{k+1}}{\Delta z_k}
+\frac{\overline\beta\alpha^2}{2A}\norm{\Delta z_k}_2^2
+A\delta_u\\
&=
A f(y)+\alpha\ip{g_{k+1}}{\Delta z_k}
+\frac12\norm{\Delta z_k}_2^2
+A\delta_u.
\end{align*}
Comparison with \eqref{eq:estimate-mid} proves \eqref{eq:estimate-induction} at $k+1$.

Evaluate $\psi_N$ at $x^\star$.  The lower oracle inequality gives
\[
f(y_i)+\ip{g_i}{x^\star-y_i}
\le f(x^\star)+\delta_\ell,
\]
so
\[
\psi_N^*
\le
\frac12\norm{x^\star-x_0}_2^2
+A_Nf(x^\star)+A_N\delta_\ell.
\]
Combining this with \eqref{eq:estimate-induction} and using $A_k\le A_N$ gives \eqref{eq:accelerated-bound}.

Finally, put $s_k:=\sqrt{\overline\beta A_k}$.  Equation \eqref{eq:alpha} implies
\[
s_{k+1}^2-s_k^2=s_{k+1},
\qquad
s_{k+1}=\frac{1+\sqrt{1+4s_k^2}}{2}\ge s_k+\frac12.
\]
Since $s_0=0$, induction gives $s_N\ge N/2$, proving \eqref{eq:A-lower}.
\end{proof}

\begin{corollary}[Finite-difference accelerated subroutine]
\label{cor:subroutine}
Let $x_0\in\cC$, let $N\ge1$, and let $\delta_g>0$ satisfy
\[
h:=\frac{2\delta_g}{\beta\sqrt d}\le\frac R2.
\]
Using \eqref{eq:fd} at every accelerated iteration yields an output $x_N\in\cC$ satisfying
\begin{equation}
\label{eq:subroutine-bound}
f(x_N)-f^\star
\le
\frac{4\beta\norm{x_0-x^\star}_2^2}{N^2}
+(N+1)R\delta_g
+\frac{N\delta_g^2}{2\beta}.
\end{equation}
The subroutine uses $(d+1)N$ exact value calls, all in $X$.
\end{corollary}

\begin{proof}
Use \cref{lem:fd} to obtain gradient error at most $\delta_g$, then combine \cref{lem:inexact-oracle,lem:accelerated-error}.  The first term in \eqref{eq:accelerated-bound} is at most $4\beta\norm{x_0-x^\star}^2/N^2$ by \eqref{eq:A-lower}.
\end{proof}

\subsection{The direct smooth-convex branch}

\begin{proposition}[Direct accuracy upper bound]
\label{prop:direct-upper}
For $0<\epsilon\le\beta R^2$,
\begin{equation}
\label{eq:direct-upper}
N_{\epsilon,\mathrm{det-ad}}^{\mathrm{sc,val}}(d,R,\mu,\beta)
\le
(d+1)\left\lceil2\sqrt{\frac{\beta R^2}{\epsilon}}\right\rceil
\le
6d\sqrt Q.
\end{equation}
\end{proposition}

\begin{proof}
Take $x_0=0$ and
\[
N:=\left\lceil2\sqrt{\frac{\beta R^2}{\epsilon}}\right\rceil,
\qquad
\delta_g:=\frac{\epsilon}{8R(N+1)}.
\]
Since $\epsilon\le\beta R^2$, the associated finite-difference step satisfies $h\le R/2$.  Also $\norm{x^\star}\le R/2$.  Therefore \eqref{eq:subroutine-bound} gives
\[
\frac{4\beta\norm{x^\star}^2}{N^2}
\le\frac{\beta R^2}{N^2}
\le\frac\epsilon4,
\]
\[
(N+1)R\delta_g=\frac\epsilon8,
\]
and
\[
\frac{N\delta_g^2}{2\beta}
=
\frac{N\epsilon^2}{128\beta R^2(N+1)^2}
\le\frac\epsilon{128}.
\]
Thus the output error is strictly smaller than $\epsilon$.  The call bound follows as in \eqref{eq:direct-upper}.
\end{proof}

\subsection{Burn-in and strongly convex restart}

\begin{proposition}[One $O(\sqrt\kappa)$ burn-in phase]
\label{prop:burnin}
Let
\[
K:=\left\lceil8\sqrt\kappa\right\rceil.
\]
Starting from $0$ and using the accelerated subroutine with
\begin{equation}
\label{eq:burnin-delta}
\delta_g^{(0)}:=\frac{\mu R}{128(K+1)}
\end{equation}
produces $x^{(0)}\in\cC$ such that
\begin{equation}
\label{eq:burnin-gap}
f(x^{(0)})-f^\star\le\Delta_0:=\frac{\mu R^2}{32}.
\end{equation}
\end{proposition}

\begin{proof}
The finite-difference step is admissible because
\[
\frac{2\delta_g^{(0)}}{\beta\sqrt d}
=\frac{R}{64\kappa(K+1)\sqrt d}
\le\frac R2.
\]
Using $\norm{x^\star}\le R/2$ and $K^2\ge64\kappa$ in \eqref{eq:subroutine-bound},
\[
\frac{4\beta\norm{x^\star}^2}{K^2}
\le\frac{\mu R^2}{64}.
\]
The two error terms satisfy
\[
(K+1)R\delta_g^{(0)}=\frac{\mu R^2}{128}
\]
and
\[
\frac{K(\delta_g^{(0)})^2}{2\beta}
=
\mu R^2\frac{K}{32768\kappa(K+1)^2}
\le\frac{\mu R^2}{32768}.
\]
Their sum is smaller than $\mu R^2/32$.
\end{proof}

\begin{proposition}[One restart halves a certified gap]
\label{prop:restart}
Suppose $x\in\cC$ satisfies
\[
f(x)-f^\star\le\Delta
\qquad\text{and}\qquad
0<\Delta\le\Delta_0.
\]
Run the accelerated subroutine for $K=\lceil8\sqrt\kappa\rceil$ iterations from $x$, using
\begin{equation}
\label{eq:restart-delta}
\delta_g:=\frac{\Delta}{16R(K+1)}.
\end{equation}
Then its output $x^+\in\cC$ satisfies
\begin{equation}
\label{eq:restart-half}
f(x^+)-f^\star\le\frac\Delta2.
\end{equation}
\end{proposition}

\begin{proof}
Strong convexity gives
\[
\norm{x-x^\star}_2^2
\le\frac{2\Delta}{\mu}.
\]
Consequently, the first term in \eqref{eq:subroutine-bound} is at most
\[
\frac{8\beta\Delta}{\mu K^2}
\le\frac\Delta8.
\]
The lower-model error contributes
\[
(K+1)R\delta_g=\frac\Delta{16}.
\]
Because $\Delta\le\mu R^2/32\le\beta R^2/32$,
\[
\frac{K\delta_g^2}{2\beta}
=
\frac{K\Delta^2}{512\beta R^2(K+1)^2}
\le\frac{\Delta}{16384}.
\]
The total is strictly below $\Delta/2$.  The finite-difference displacement is also admissible since
\[
\frac{2\delta_g}{\beta\sqrt d}
\le\frac{R}{256(K+1)\sqrt d}
\le\frac R2.
\]
\end{proof}

\begin{proposition}[Restarted strongly convex upper bound]
\label{prop:sc-upper}
Let
\[
S:=\max\left\{0,\left\lceil\log_2\frac{\mu R^2}{32\epsilon}\right\rceil\right\}.
\]
Then
\begin{equation}
\label{eq:sc-upper-explicit}
N_{\epsilon,\mathrm{det-ad}}^{\mathrm{sc,val}}(d,R,\mu,\beta)
\le
(d+1)\left\lceil8\sqrt\kappa\right\rceil(1+S).
\end{equation}
In particular,
\begin{equation}
\label{eq:sc-upper-asymptotic}
N_{\epsilon,\mathrm{det-ad}}^{\mathrm{sc,val}}(d,R,\mu,\beta)
\le
C d\sqrt\kappa
\left[1+\logp{\frac{\mu R^2}{\epsilon}}\right]
\end{equation}
for a universal constant $C$.
\end{proposition}

\begin{proof}
Apply \cref{prop:burnin}.  If $\epsilon\ge\Delta_0$, stop.  Otherwise apply \cref{prop:restart} successively with certified gaps
\[
\Delta_s:=2^{-s}\Delta_0,
\qquad s=0,\ldots,S-1.
\]
After $S$ restarts the error is at most $2^{-S}\Delta_0\le\epsilon$.  Each phase uses $(d+1)K$ calls, proving \eqref{eq:sc-upper-explicit}; the asymptotic form follows from $d+1\le2d$, $K\le9\sqrt\kappa$, and the definition of $S$.
\end{proof}

\begin{proof}[Proof of \cref{thm:upper-main}]
Run the better of the algorithms in \cref{prop:direct-upper,prop:sc-upper}.  Since
\[
\frac{\mu R^2}{\epsilon}=\frac Q\kappa,
\]
their bounds give \eqref{eq:upper-main} after adjusting the universal constant.
\end{proof}

\section{Architecture of the lower bound}
\label{sec:lower-architecture}
\label{sec:overview}

This part of the proof establishes the hard-instance compiler in \cref{thm:compiler}.  The construction has a scalar chain length $m$ and a robustness radius $\rho$.  A biased maximum
\[
g_{m,\rho}(y)=\max_{1\le i\le m}\{y_i-8\rho i\}
\]
is smoothed by a Moreau envelope $h_{m,\rho}$.  Its dual problem is a concave quadratic maximization over the simplex.  When all coordinates from index $i$ onward lie in $[-\rho,\rho]$, moving any dual mass from a later coordinate to coordinate $i$ strictly improves the dual objective.  Hence every dual optimizer is supported on the first $i$ coordinates, and the smoothed value is exactly independent of the remaining tail.

The $m$ chain coordinates are embedded through an orthonormal frame $U=(u_1,\ldots,u_m)$.  Queries are grouped into blocks of size
\[
b=\left\lfloor\frac d4\right\rfloor.
\]
All replies in block $j$ are generated before $u_j$ is selected.  We then choose $u_j$ orthogonal to every point of the current block and with projection at most $\rho$ on every earlier query.  The current coordinate is therefore exactly zero, while every future coordinate is eventually made small.  Exact shielding makes the online reply independent of all unselected directions.

After the final query, one additional sentinel direction is selected orthogonal to the terminal output.  The remaining directions are chosen with small projection on all queries and the output.  Consequently the output has not reached the final chain coordinate.  A comparison point with all chain coordinates equal to $-R_0/\sqrt m$ then yields objective gap $\Omega(\beta_0 R^2/m^2)$.

The geometric cost is the simultaneous small-projection requirement.  Since all queries have norm at most $R$, a spherical-cap union bound finds the required direction whenever
\[
d\rho^2/R^2\gtrsim\log(d^2).
\]
The optimization calibration uses $\rho\asymp R/m^{3/2}$, giving the condition $m^3\log(ed)\lesssim d$.

\section{A smooth robust value chain}
\label{sec:chain}

Fix $m\ge1$ and $\rho>0$.  Define
\begin{equation}
\label{eq:gdef}
g_{m,\rho}(y):=\max_{1\le i\le m}\{y_i-8\rho i\},
\qquad y\in\R^m,
\end{equation}
and
\begin{equation}
\label{eq:hdef}
h_{m,\rho}(y)
:=
\inf_{z\in\R^m}
\left\{
 g_{m,\rho}(z)+\frac1\rho\norm{z-y}_2^2
\right\}.
\end{equation}
This is the Moreau envelope with parameter $\rho/2$.

\subsection{Smoothness and dual representation}

\begin{lemma}[Basic Moreau properties]
\label{lem:moreau}
The function $g_{m,\rho}$ is convex and $1$-Lipschitz.  The function $h_{m,\rho}$ is convex and continuously differentiable, and
\begin{align}
0&\le g_{m,\rho}(y)-h_{m,\rho}(y)\le\frac\rho4,
\label{eq:approx}\\
\norm{\nabla h_{m,\rho}(y)}_2&\le1,
\label{eq:hgrad}\\
\Lip(\nabla h_{m,\rho})&\le\frac2\rho.
\label{eq:hsmooth}
\end{align}
If $z_y$ is the unique minimizer in \eqref{eq:hdef}, then
\begin{equation}
\label{eq:proxradius}
\norm{z_y-y}_2\le\frac\rho2.
\end{equation}
\end{lemma}

\begin{proof}
The function $g_{m,\rho}$ is the maximum of affine functions with slopes $e_1,\ldots,e_m$, hence it is convex and $1$-Lipschitz.  Choosing $z=y$ in \eqref{eq:hdef} gives $h(y)\le g(y)$.  Conversely, with $r=\norm{z-y}_2$,
\[
g(z)+\frac{r^2}{\rho}
\ge g(y)-r+\frac{r^2}{\rho}
\ge g(y)-\frac\rho4,
\]
because the minimum of $-r+r^2/\rho$ over $r\ge0$ is $-\rho/4$.  This proves \eqref{eq:approx}.

The objective in \eqref{eq:hdef} is strongly convex in $z$, so $z_y$ is unique.  First-order optimality gives an $s_y\in\partial g(z_y)$ such that
\[
s_y+\frac2\rho(z_y-y)=0.
\]
Every subgradient of a $1$-Lipschitz convex function has norm at most one, which proves \eqref{eq:proxradius}.  Standard Moreau-envelope calculus gives
\[
\nabla h(y)=\frac2\rho(y-z_y).
\]
Thus \eqref{eq:hgrad} follows, and firm nonexpansiveness of the proximal map gives \eqref{eq:hsmooth}; see, e.g., \citet{bauschke2017convex}.
\end{proof}

Let
\[
\Delta_m:=\left\{p\in\R_+^m:\sum_{i=1}^m p_i=1\right\}.
\]

\begin{lemma}[Simplex dual]
\label{lem:dual}
For every $y\in\R^m$,
\begin{equation}
\label{eq:dual}
h_{m,\rho}(y)
=
\max_{p\in\Delta_m}
\left\{
\sum_{i=1}^m p_i(y_i-8\rho i)
-\frac\rho4\norm p_2^2
\right\}.
\end{equation}
\end{lemma}

\begin{proof}
Because the maximum of finitely many scalars equals the maximum over their convex combinations,
\[
g_{m,\rho}(z)
=
\max_{p\in\Delta_m}
\sum_{i=1}^m p_i(z_i-8\rho i).
\]
The simplex is compact and convex; the displayed expression plus $\rho^{-1}\norm{z-y}^2$ is affine in $p$ and strongly convex in $z$.  Minimax interchange is therefore valid.  For fixed $p$, the minimizer over $z$ is
\[
z=y-\frac\rho2p,
\]
and the minimum equals
\[
\sum_i p_i(y_i-8\rho i)-\frac\rho4\norm p_2^2.
\]
Maximizing over $p\in\Delta_m$ proves the claim.
\end{proof}

\subsection{Exact shielding}

Define the robust progress index
\begin{equation}
\label{eq:index}
i_\rho^+(y)
:=
\min\left\{
 i\in[m]: |y_k|\le\rho\ \text{for every }k\ge i
\right\},
\end{equation}
with $i_\rho^+(y)=m+1$ if the set is empty.

\begin{lemma}[Exact prefix shielding]
\label{lem:shield}
Fix $j\in[m]$.  If
\begin{equation}
\label{eq:shield-window}
|y_k|\le\rho
\qquad\text{for every }k\ge j,
\end{equation}
then every maximizer of the dual problem \eqref{eq:dual} is supported on
$\{1,\ldots,j\}$, and
\begin{equation}
\label{eq:shield}
h_{m,\rho}(y)
=
h_{m,\rho}(y_1,\ldots,y_j,0,\ldots,0).
\end{equation}
In particular, if $y_j=0$, then
\begin{equation}
\label{eq:shield-zero}
h_{m,\rho}(y)
=
h_{m,\rho}(y_1,\ldots,y_{j-1},0,\ldots,0).
\end{equation}
Consequently, if $i=i_\rho^+(y)\le m$, then the value is exactly
independent of $y_{i+1},\ldots,y_m$.
\end{lemma}

\begin{proof}
Set $c_k=y_k-8\rho k$.  For every $k>j$, condition
\eqref{eq:shield-window} gives
\[
c_j-c_k
\ge -\rho-8\rho j-(\rho-8\rho k)
=8\rho(k-j)-2\rho
\ge6\rho.
\]
Let $p$ maximize the dual problem \eqref{eq:dual}.  Suppose that
$p_k>0$ for some $k>j$.  Form $\widetilde p\in\Delta_m$ by moving all
mass $p_k$ from coordinate $k$ to coordinate $j$:
\[
\widetilde p_j=p_j+p_k,
\qquad
\widetilde p_k=0,
\qquad
\widetilde p_\ell=p_\ell\quad(\ell\notin\{j,k\}).
\]
The change in the dual objective is
\begin{align*}
&p_k(c_j-c_k)
-\frac\rho4\big[(p_j+p_k)^2-p_j^2-p_k^2\big]\\
&\qquad
=p_k(c_j-c_k)-\frac\rho2p_jp_k
\ge p_k\left(6\rho-\frac\rho2\right)>0,
\end{align*}
a contradiction.  Thus every dual maximizer is supported on
$\{1,\ldots,j\}$.

After replacing $y_{j+1},\ldots,y_m$ by zero, the same argument again
forces every dual maximizer to be supported on $\{1,\ldots,j\}$.  On
that face of the simplex, the two dual objectives coincide, which proves
\eqref{eq:shield}.  If $y_j=0$, the right-hand side of
\eqref{eq:shield} is exactly the vector displayed in
\eqref{eq:shield-zero}.  The final assertion follows by taking
$j=i_\rho^+(y)$.
\end{proof}

\begin{remark}
The proof is global and value-level.  It does not merely show that future affine pieces are inactive near one proximal point; it identifies the support of every optimizer of the exact dual problem.  This distinction is what makes the later fixed-transcript argument rigorous.
\end{remark}

\subsection{A progress-limited gap}

Set
\begin{equation}
\label{eq:R0rho}
R_0:=\frac R4,
\qquad
\rho:=\frac{R_0}{64m^{3/2}}
=\frac{R}{256m^{3/2}},
\end{equation}
and define
\begin{equation}
\label{eq:ycirc}
y^\circ:=-\frac{R_0}{\sqrt m}\1_m.
\end{equation}
Then $\norm{y^\circ}_2=R_0$.

\begin{lemma}[Gap before the final coordinate]
\label{lem:progressgap}
If $i_\rho^+(y)\le m$, then
\begin{equation}
\label{eq:progressgap}
h_{m,\rho}(y)-h_{m,\rho}(y^\circ)
\ge\frac{3R_0}{4\sqrt m}.
\end{equation}
\end{lemma}

\begin{proof}
The increasing bias implies
\[
g_{m,\rho}(y^\circ)
=-\frac{R_0}{\sqrt m}-8\rho.
\]
The condition $i_\rho^+(y)\le m$ implies $|y_m|\le\rho$, and hence
\[
g_{m,\rho}(y)
\ge y_m-8m\rho
\ge-\rho-8m\rho.
\]
Using \eqref{eq:approx},
\begin{align*}
h(y)-h(y^\circ)
&\ge g(y)-\frac\rho4-g(y^\circ)\\
&\ge\frac{R_0}{\sqrt m}-8m\rho+\frac{27\rho}{4}\\
&\ge\frac{R_0}{\sqrt m}-8m\rho.
\end{align*}
By \eqref{eq:R0rho}, $8m\rho=R_0/(8\sqrt m)$.  Thus the left-hand side is at least $7R_0/(8\sqrt m)$, which implies \eqref{eq:progressgap}.
\end{proof}

\section{Avoiding spherical caps in a subspace}
\label{sec:cap}

The delayed rotation requires a direction that is exactly orthogonal to the current block and has small projection on all earlier points.  The following elementary lemma supplies it.

\begin{lemma}[Subspace cap avoidance]
\label{lem:cap}
Let $S\subseteq\R^d$ be a $p$-dimensional subspace and let $a_1,\ldots,a_M\in B_2^d(R)$.  If $0<\rho\le R$ and
\begin{equation}
\label{eq:capcond}
p\rho^2
\ge
64R^2\log\bigl(4(M+1)\bigr),
\end{equation}
then there exists a unit vector $u\in S$ such that
\begin{equation}
\label{eq:capconclusion}
|\ip{u}{a_i}|\le\rho,
\qquad i=1,\ldots,M.
\end{equation}
\end{lemma}

\begin{proof}
The case $M=0$ is immediate.  Let $g$ be a standard Gaussian vector in $S\cong\R^p$ and put $u=g/\norm g_2$.  For fixed $a_i$, only $P_Sa_i$ matters.  On the event $\norm g_2\ge\sqrt{p/2}$, the implication
\[
|\ip{u}{a_i}|>\rho
\quad\Longrightarrow\quad
\left|\ip{g}{\frac{P_Sa_i}{\norm{P_Sa_i}_2}}\right|
>\frac{\rho}{R}\sqrt{\frac p2}
\]
holds whenever $P_Sa_i\ne0$.  A Gaussian tail bound and the standard lower-tail estimate for a chi-square variable give
\[
\Prob\{|\ip{u}{a_i}|>\rho\}
\le
2\exp\left(-\frac{p\rho^2}{4R^2}\right)
+\exp\left(-\frac p{16}\right)
\le
3\exp\left(-\frac{p\rho^2}{16R^2}\right),
\]
where the last inequality uses $\rho\le R$.  A union bound yields
\[
\Prob\left\{\max_{1\le i\le M}|\ip{u}{a_i}|>\rho\right\}
\le
3M\exp\left(-\frac{p\rho^2}{16R^2}\right)<1
\]
under \eqref{eq:capcond}.  Therefore a desired unit vector exists.  See \citet{vershynin2018high} for standard Gaussian and spherical concentration estimates.
\end{proof}

\begin{corollary}[A convenient parameter condition]
\label{cor:capparameter}
There is a universal $c_0>0$ such that the following holds.  Suppose $d\ge8$, $m\le d/4$, $M\le d^2/16+1$, $p\ge d/2$, and $\rho$ is given by \eqref{eq:R0rho}.  If
\[
m^3\log(ed)\le c_0d,
\]
then \eqref{eq:capcond} holds.  One may take, for example, $c_0=2^{-24}$ after increasing the harmless dimension threshold.
\end{corollary}

\begin{proof}
Equation \eqref{eq:R0rho} gives
\[
\frac{\rho^2}{R^2}=\frac{1}{2^{16}m^3},
\qquad
\frac{p\rho^2}{R^2}\ge\frac{d}{2^{17}m^3}.
\]
For $d\ge8$ and $M\le d^2/16+1$,
\[
\log(4(M+1))\le2\log(ed).
\]
The conclusion follows from \eqref{eq:capcond} with the displayed conservative value of $c_0$.
\end{proof}

\section{The batched delayed-rotation compiler}
\label{sec:compiler}

Fix a deterministic adaptive algorithm and a total budget $T$.  Let
\begin{equation}
\label{eq:blocksize}
b:=\left\lfloor\frac d4\right\rfloor,
\qquad
T=qb+s,
\qquad
0\le s<b.
\end{equation}
Assume
\begin{equation}
\label{eq:Tless}
T<bm,
\end{equation}
so $0\le q\le m-1$.  We construct orthonormal directions
\[
u_1,\ldots,u_m\in\R^d,
\qquad
U=(u_1,\ldots,u_m)\in\R^{d\times m}.
\]
Define
\begin{equation}
\label{eq:Aeta}
A:=\frac{\beta_0\rho}{8},
\qquad
\eta:=\frac{\beta_0}{4096m^2},
\end{equation}
and the direction-independent term
\begin{equation}
\label{eq:q0}
q_0(x)
:=
\frac{\beta_0}{8}\dist\bigl(x,B_2^d(R_0)\bigr)^2
+\frac\eta2\norm{x}_2^2.
\end{equation}

\subsection{Complete blocks}

Suppose that $u_1,\ldots,u_{j-1}$ have been selected before complete block $j$.  Every query $x$ in this block receives the response
\begin{equation}
\label{eq:fullresponse}
r(x)
:=
A h_{m,\rho}\bigl(
\ip{u_1}{x},\ldots,\ip{u_{j-1}}{x},0,\ldots,0
\bigr)+q_0(x).
\end{equation}
The algorithm may adapt within the block, but \eqref{eq:fullresponse} does not depend on $u_j$.  Consequently all points
\[
x_{j,1},\ldots,x_{j,b}
\]
are determined before $u_j$ is selected.

Set
\begin{equation}
\label{eq:Sj}
S_j
:=
\Span\{u_1,\ldots,u_{j-1},x_{j,1},\ldots,x_{j,b}\}^{\perp}.
\end{equation}
Since $j\le q\le m-1$ and $m\le d/4$,
\begin{equation}
\label{eq:Sjdim}
\dim S_j
\ge d-(j-1)-b
\ge\frac d2.
\end{equation}
Apply \cref{lem:cap} inside $S_j$ to all queries from earlier blocks and select a unit $u_j\in S_j$ such that
\begin{align}
|\ip{u_j}{x_t}|&\le\rho
&&\text{for every earlier query }x_t,
\label{eq:oldsmall}\\
\ip{u_j}{x_{j,\ell}}&=0
&&\text{for }\ell=1,\ldots,b.
\label{eq:currentzero}
\end{align}
The second relation is exact because $u_j\in S_j$.

\subsection{The incomplete block, the output, and completion}

After the $q$ complete blocks, answer the remaining $s$ queries by
\begin{equation}
\label{eq:lastresponse}
r(x)
:=
A h_{m,\rho}\bigl(
\ip{u_1}{x},\ldots,\ip{u_q}{x},0,\ldots,0
\bigr)+q_0(x).
\end{equation}
Once all responses are returned, the deterministic terminal output $\widehat x$ is fixed.  Select $u_{q+1}$ in
\[
\Span\{u_1,\ldots,u_q,x_{qb+1},\ldots,x_T,\widehat x\}^{\perp}
\]
so that it also has projection at most $\rho$ on every query from the complete blocks.  This is possible by \cref{lem:cap}: the displayed orthogonal complement has dimension at least $d/2$, because $q\le m-1$, $s+1\le b$, and $m,b\le d/4$.  Thus
\begin{align}
\ip{u_{q+1}}{x_t}&=0,
&&t=qb+1,\ldots,T,
\label{eq:sentquery}\\
\ip{u_{q+1}}{\widehat x}&=0,
\label{eq:sentout}\\
|\ip{u_{q+1}}{x_t}|&\le\rho,
&&t\le qb.
\label{eq:sentold}
\end{align}

Finally, for $k=q+2,\ldots,m$, select $u_k$ successively in
\[
\Span\{u_1,\ldots,u_{k-1}\}^{\perp}
\]
so that
\begin{equation}
\label{eq:futuresmall}
|\ip{u_k}{x_t}|\le\rho
\quad(t=1,\ldots,T),
\qquad
|\ip{u_k}{\widehat x}|\le\rho.
\end{equation}
The available subspace has dimension at least $d-m+1\ge3d/4$.  Every cap-avoidance step controls at most $T+1\le d^2/16+1$ vectors under \eqref{eq:Tless} and $m\le d/4$.  Hence \cref{cor:capparameter} validates all selections.

\begin{remark}[Fixed horizon and early stopping]
A pathwise algorithm that stops in fewer than $T$ calls may be padded by ignored repeated queries.  Therefore it is enough to prove the lower bound for fixed horizons.  The construction above never changes a response after it has been returned.
\end{remark}

\section{One fixed objective realizes the transcript}
\label{sec:fixed}

Once the frame $U$ is complete, define
\begin{equation}
\label{eq:FU}
F_U(x)
:=
A h_{m,\rho}(U^\top x)
+
\frac{\beta_0}{8}\dist\bigl(x,B_2^d(R_0)\bigr)^2
+
\frac\eta2\norm{x}_2^2.
\end{equation}

\begin{proposition}[Exact fixed-objective consistency]
\label{prop:transcript}
Every response generated in \cref{sec:compiler} equals the corresponding value of the single function \eqref{eq:FU}:
\[
r_t=F_U(x_t),
\qquad t=1,\ldots,T.
\]
Consequently, the query sequence and terminal output generated by the online construction are exactly the query sequence and output produced when the algorithm interacts with the fixed objective $F_U$.
\end{proposition}

\begin{proof}
Consider $x=x_{j,\ell}$ in complete block $j$ and write $y=U^\top x$.  By \eqref{eq:currentzero}, $y_j=0$.  Every direction $u_k$ with $k>j$ is selected after this query and is required, either during a later complete block or during final completion, to satisfy $|y_k|\le\rho$.  Thus \eqref{eq:shield-window} holds at the fixed prefix $j$, and \eqref{eq:shield-zero} gives
\[
h_{m,\rho}(U^\top x)
=
h_{m,\rho}\bigl(
\ip{u_1}{x},\ldots,\ip{u_{j-1}}{x},0,\ldots,0
\bigr),
\]
which is precisely the hard-core term in \eqref{eq:fullresponse}.

For a query in the incomplete block, \eqref{eq:sentquery} gives a zero at coordinate $q+1$, while \eqref{eq:futuresmall} controls every later coordinate.  Applying \eqref{eq:shield-zero} with $j=q+1$ turns \eqref{eq:FU} into \eqref{eq:lastresponse}.

Finally, use induction over time.  If all earlier replies agree with $F_U$, determinism gives the same current query; the preceding argument gives the same current reply.  Hence the complete transcript and the terminal output agree with the true interaction with $F_U$.
\end{proof}

The terminal completion also gives
\begin{equation}
\label{eq:outputprogress}
i_\rho^+(U^\top\widehat x)
\le q+1\le m,
\end{equation}
because coordinate $q+1$ is zero by \eqref{eq:sentout} and every later coordinate has magnitude at most $\rho$ by \eqref{eq:futuresmall}.

\section{Admissibility of the hard objective}
\label{sec:admissibility}
\label{sec:admissible}

\begin{proposition}[Convexity, smoothness, and exact strong-convexity modulus]
\label{prop:smooth}
The function $F_U$ is convex and continuously differentiable.  Moreover,
\[
\Lip(\nabla F_U)<\beta_0,
\qquad
\mu(F_U)=\eta=\frac{\beta_0}{4096m^2},
\]
where $\mu(F_U)$ denotes the largest strong-convexity modulus of $F_U$.
In particular, $F_U$ has a unique global minimizer.
\end{proposition}

\begin{proof}
Each term in \eqref{eq:FU} is convex and continuously differentiable.  Since $U^\top U=I_m$, \eqref{eq:hsmooth} gives
\[
\Lip\bigl(\nabla[A h_{m,\rho}(U^\top\cdot)]\bigr)
\le A\frac2\rho
=\frac{\beta_0}{4}.
\]
For a nonempty closed convex set $C$, the function $x\mapsto\frac12\dist(x,C)^2$ has gradient $x-\Pi_Cx$, which is $1$-Lipschitz.  Therefore the second term in \eqref{eq:FU} has gradient-Lipschitz constant $\beta_0/4$.  The final quadratic has gradient-Lipschitz constant $\eta$.  Thus
\[
\Lip(\nabla F_U)
\le\frac{\beta_0}{4}+\frac{\beta_0}{4}+\eta<\beta_0.
\]
The quadratic term makes $F_U$ at least $\eta$-strongly convex and coercive, so a unique minimizer exists.

It remains to show that the modulus is not larger.  Since $m<d$, choose a unit vector $v\in\operatorname{range}(U)^\perp$.  For every $|t|\le R_0$,
\[
U^\top(tv)=0,
\qquad
\dist(tv,B_2^d(R_0))=0,
\]
and hence
\[
F_U(tv)=A h_{m,\rho}(0)+\frac\eta2t^2.
\]
If $F_U$ were $\mu$-strongly convex for some $\mu>\eta$, the midpoint form of strong convexity applied to $tv$ and $-tv$ would imply
\[
F_U(tv)+F_U(-tv)-2F_U(0)\ge\mu t^2.
\]
The left-hand side equals $\eta t^2$, a contradiction for $t\ne0$.  Therefore $\mu(F_U)=\eta$.
\end{proof}

\begin{proposition}[The minimizer lies in the interior promise]
\label{prop:minloc}
Let $x_U^\star$ be the unique minimizer of $F_U$.  Then
\[
\norm{x_U^\star}_2<\frac R2.
\]
\end{proposition}

\begin{proof}
The hard-core gradient has norm at most $A$ by \eqref{eq:hgrad} and the isometry of $U$.  Write $r=\norm{x_U^\star}_2$.  If $r\le R_0$, there is nothing to prove.  Otherwise let $\nu=x_U^\star/r$.  Outside $B_2^d(R_0)$, the gradient of the distance term is
\[
\frac{\beta_0}{4}(r-R_0)\nu.
\]
Taking the inner product of the first-order condition $\nabla F_U(x_U^\star)=0$ with $\nu$ yields
\[
0
\ge
-A+\frac{\beta_0}{4}(r-R_0)+\eta r.
\]
Therefore
\[
r\le R_0+\frac{4A}{\beta_0}=R_0+\frac\rho2.
\]
Since $R_0=R/4$ and $\rho\le R_0/64$,
\[
r\le R_0+\frac{R_0}{128}<\frac R2.
\]
\end{proof}

\begin{remark}
The distance-squared term is used only to keep the global minimizer in the promised interior ball.  The final small quadratic supplies uniqueness and strong convexity.  Neither term depends on the hidden frame.
\end{remark}

\section{Terminal error and the integer theorem}
\label{sec:gap}

Define the comparison point
\begin{equation}
\label{eq:xcirc}
x^\circ:=Uy^\circ.
\end{equation}
Because $U$ is an isometry on $\R^m$,
\[
\norm{x^\circ}_2=\norm{y^\circ}_2=R_0,
\qquad
U^\top x^\circ=y^\circ.
\]
Thus the distance penalty vanishes at $x^\circ$, and $x^\circ\in X$.

\begin{proposition}[Output gap]
\label{prop:outputgap}
The fixed objective and terminal point constructed above satisfy
\begin{equation}
\label{eq:outputgap}
F_U(\widehat x)-F_U^\star
\ge
\frac{11}{2^{17}}\frac{\beta_0 R^2}{m^2}
\ge
2^{-17}\frac{\beta_0 R^2}{m^2}.
\end{equation}
\end{proposition}

\begin{proof}
By \eqref{eq:outputprogress} and \cref{lem:progressgap},
\[
h_{m,\rho}(U^\top\widehat x)-h_{m,\rho}(y^\circ)
\ge\frac{3R_0}{4\sqrt m}.
\]
Since $F_U^\star\le F_U(x^\circ)$, and the radial and quadratic terms at $\widehat x$ are nonnegative,
\begin{align*}
F_U(\widehat x)-F_U^\star
&\ge F_U(\widehat x)-F_U(x^\circ)\\
&\ge
A\frac{3R_0}{4\sqrt m}-\frac\eta2R_0^2.
\end{align*}
Using \eqref{eq:R0rho} and \eqref{eq:Aeta},
\[
A\frac{3R_0}{4\sqrt m}
=
\frac{3\beta_0 R_0^2}{2048m^2},
\qquad
\frac\eta2R_0^2
=
\frac{\beta_0 R_0^2}{8192m^2}.
\]
Their difference is
\[
\frac{11\beta_0 R_0^2}{8192m^2}
=
\frac{11}{2^{17}}\frac{\beta_0 R^2}{m^2},
\]
which proves the proposition.
\end{proof}

\begin{proof}[Proof of \cref{thm:compiler}]
For $d\ge8$,
\[
b=\left\lfloor\frac d4\right\rfloor\ge\frac d8.
\]
Thus $T<dm/8$ implies $T<bm$, so the construction in \cref{sec:compiler} applies.  The condition $m^3\log(ed)\le c_0d$ validates every cap-avoidance step by \cref{cor:capparameter}.  Exact fixed-objective consistency follows from \cref{prop:transcript}; membership in $\cG_{\beta_0}(d,R)$ follows from \cref{prop:smooth,prop:minloc}; and the terminal error is \eqref{eq:outputgap}.  The exact strong-convexity modulus is the value of $\eta$ in \eqref{eq:Aeta}, by \cref{prop:smooth}.
\end{proof}

\section{Calibration to fixed \texorpdfstring{$(\mu,\beta,\epsilon)$}{(mu,beta,epsilon)}}
\label{sec:calibration}
\label{sec:rate}

The hard compiler uses an auxiliary smoothness parameter $\beta_0$ and produces exact strong convexity $\beta_0/(4096m^2)$.  We now choose these quantities so that the final objective belongs to the public class $\cF_{\mu,\beta}(d,R)$.

\begin{lemma}[A hidden-shift quadratic]
\label{lem:quadratic}
Assume $0<\epsilon\le\beta R^2/64$.  Every deterministic adaptive algorithm using at most $d-2$ value calls admits a hard function in $\cF_{\mu,\beta}(d,R)$ whose terminal error is at least $2\epsilon$.
\end{lemma}

\begin{proof}
Set
\[
\lambda:=\max\left\{\mu,\frac{64\epsilon}{R^2}\right\}.
\]
The accuracy assumption and $\mu\le\beta$ imply $\lambda\le\beta$.  Put $r=R/4$.  During the interaction answer every query $x_t$ by
\begin{equation}
\label{eq:quadratic-reply}
r_t:=\frac\lambda2\left(\norm{x_t}_2^2+r^2\right).
\end{equation}
After the $T\le d-2$ queries have been generated and the deterministic output $\widehat x$ is fixed, choose a unit vector
\[
v\perp\Span\{x_1,\ldots,x_T,\widehat x\}.
\]
Such a vector exists because the displayed span has dimension at most $T+1<d$.  Define
\[
F_v(x):=\frac\lambda2\norm{x-rv}_2^2.
\]
This function is $\lambda$-strongly convex and $\lambda$-smooth, hence it belongs to $\cF_{\mu,\beta}(d,R)$, and its minimizer $rv$ lies in $B_2^d(R/2)$.  Orthogonality gives $F_v(x_t)=r_t$ for every query, so one fixed objective realizes the transcript.  It also gives
\[
F_v(\widehat x)-F_v^\star
=\frac\lambda2\left(\norm{\widehat x}_2^2+r^2\right)
\ge\frac{\lambda R^2}{32}
\ge2\epsilon.
\]
\end{proof}

\begin{proof}[Proof of \cref{thm:lower-main}]
It is enough to prove the theorem with $c_\epsilon=1/64$.  Set
\[
M:=\min\left\{
\sqrt Q,
\sqrt\kappa,
D_d
\right\}.
\]
Let $c_0$ be the compiler constant in \cref{thm:compiler}.  Choose a universal $a>0$ so that
\begin{equation}
\label{eq:a-choice-sc}
a\le\frac14,
\qquad
a^3\le c_0,
\qquad
2^{18}a^2\le1.
\end{equation}
Increase $d_0$ so that
\[
aD_d\le\frac d4
\]
for all $d\ge d_0$.

First suppose $M\ge2/a$ and define
\[
m:=\lfloor aM\rfloor.
\]
Then
\begin{equation}
\label{eq:m-compare-sc}
\frac a2M\le m\le aM.
\end{equation}
Set
\begin{equation}
\label{eq:lambda-beta0}
\lambda:=\max\left\{\mu,\frac{64\epsilon}{R^2}\right\},
\qquad
\beta_0:=4096\lambda m^2.
\end{equation}
Since $m\le a\sqrt\kappa$ and $m\le a\sqrt Q$,
\begin{align*}
\frac{\beta_0}{\beta}
&=
4096\max\left\{\frac{m^2}{\kappa},\frac{64m^2}{Q}\right\}\\
&\le 2^{18}a^2\le1.
\end{align*}
Thus $\beta_0\le\beta$.  Moreover,
\[
m^3\log(ed)
\le a^3D_d^3\log(ed)
=a^3d
\le c_0d,
\qquad
m\le\frac d4.
\]
Apply \cref{thm:compiler} with auxiliary smoothness $\beta_0$.  The resulting fixed objective has smoothness below $\beta_0\le\beta$ and exact strong-convexity modulus
\[
\frac{\beta_0}{4096m^2}=\lambda\ge\mu.
\]
Hence it lies in $\cF_{\mu,\beta}(d,R)$.  Its terminal gap is at least
\[
2^{-17}\frac{\beta_0R^2}{m^2}
=
\frac{\lambda R^2}{32}
\ge2\epsilon.
\]
Therefore every budget $T<dm/8$ fails, and \eqref{eq:m-compare-sc} gives
\[
N_{\epsilon,\mathrm{det-ad}}^{\mathrm{sc,val}}
\ge\frac{dm}{8}
\ge\frac a{16}dM.
\]

Now suppose $M<2/a$.  By \cref{lem:quadratic},
\[
N_{\epsilon,\mathrm{det-ad}}^{\mathrm{sc,val}}
\ge d-1\ge\frac d2.
\]
Since $M<2/a$, this implies
\[
N_{\epsilon,\mathrm{det-ad}}^{\mathrm{sc,val}}
\ge\frac a4dM.
\]
Combining the two cases and taking $c=a/16$ proves \eqref{eq:lower-main}.
\end{proof}

\begin{proof}[Proof of \cref{thm:near-region}]
Condition \eqref{eq:near-condition} is exactly
\[
\sqrt{\min\{Q,\kappa\}}\le D_d.
\]
Therefore \cref{thm:lower-main} gives the lower inequality in \eqref{eq:near-main}.  For the upper inequality, if $Q\le\kappa$, use the direct branch of \cref{thm:upper-main}; then $\logp{Q/\kappa}=0$.  If $Q>\kappa$, use the restarted branch.  In either case the result is \eqref{eq:near-main}.
\end{proof}

\begin{proof}[Proof of \cref{cor:accuracy-branch}]
Here $\min\{Q,\kappa\}=Q$ and $\logp{Q/\kappa}=0$.  Apply \cref{thm:near-region}.
\end{proof}

\begin{proof}[Proof of \cref{cor:kappa-branch}]
Choose the universal $c_\epsilon$ no larger than the one in \cref{thm:lower-main} and no larger than $1$.  The condition $\epsilon\le c_\epsilon\mu R^2$ implies $Q\ge\kappa$, while the condition-number assumption implies $\kappa\le D_d^2$.  Thus \cref{thm:near-region} applies with $\min\{Q,\kappa\}=\kappa$ and
\[
\logp{Q/\kappa}
=\log\frac{\mu R^2}{\epsilon}.
\]
This proves \eqref{eq:kappa-branch}.
\end{proof}

\section{The resulting complexity landscape}
\label{sec:landscape}

The bounds can be summarized by the following table.  Logarithmic factors refer only to relative accuracy unless otherwise stated.

\begin{center}
\begin{tabular}{@{}p{0.26\textwidth}p{0.25\textwidth}p{0.25\textwidth}p{0.16\textwidth}@{}}
\toprule
Region & Lower bound & Upper bound & Status\\
\midrule
$Q\le\kappa$ and $Q\le D_d^2$
& $\Omega(d\sqrt Q)$
& $O(d\sqrt Q)$
& Optimal\\[1mm]
$\kappa\le Q$ and $\kappa\le D_d^2$
& $\Omega(d\sqrt\kappa)$
& $O(d\sqrt\kappa[1+\log(Q/\kappa)])$
& Near-optimal\\[1mm]
$\min\{Q,\kappa\}>D_d^2$
& $\Omega(dD_d)$
& \cref{thm:upper-main}
& Open\\
\bottomrule
\end{tabular}
\end{center}

\subsection{Why the high-accuracy result is not restricted to constant precision}

The lower bound $\Omega(d\sqrt\kappa)$ remains valid for every
\[
0<\epsilon\le c_\epsilon\mu R^2,
\]
not only for $\epsilon$ proportional to $\mu R^2$.  The upper bound pays one $O(d\sqrt\kappa)$ burn-in phase and then one such phase for each constant-factor error reduction.  Hence arbitrary high accuracy costs only the relative-accuracy logarithm
\[
1+\log\frac{\mu R^2}{\epsilon}.
\]
In the conventional notation that suppresses polylogarithmic factors, this is precisely a $\widetilde\Theta(d\sqrt\kappa)$ characterization.

\subsection{No unnecessary \texorpdfstring{$\log\kappa$}{log-kappa} term}

A direct application of a linear-convergence estimate from the initial gap $O(\beta R^2)$ can produce
\[
O\left(d\sqrt\kappa\log\frac{\beta R^2}{\epsilon}\right)
=
O\left(d\sqrt\kappa\left[\log\kappa+\log\frac{\mu R^2}{\epsilon}\right]\right).
\]
The burn-in argument removes this artifact.  A single accelerated convex phase reaches error $O(\mu R^2)$ in $O(d\sqrt\kappa)$ calls, after which restart depends only on $\mu R^2/\epsilon$.

\subsection{The dimension threshold}

The lower construction uses a robustness scale
\[
\rho=\Theta\left(\frac{R}{m^{3/2}}\right).
\]
The spherical-cap step requires $d\rho^2/R^2\gtrsim\log(d^2)$, leading to
\[
m^3\log(ed)\lesssim d.
\]
For the strongly convex branch $m\asymp\sqrt\kappa$, this becomes
\[
\kappa\lesssim\left(\frac{d}{\log(ed)}\right)^{2/3}.
\]
This threshold is a limitation of the present shielding geometry, not evidence that the true minimax complexity changes there.

\subsection{What remains open}

Four gaps are not closed by the present argument.
\begin{enumerate}[label=(\roman*)]
\item For nondegenerate condition numbers bounded away from $1$, the lower bound lacks the factor $\log(\mu R^2/\epsilon)$ in the high-accuracy strongly convex branch.  Closing this gap may require a multiscale no-reuse or direct-product mechanism beyond the single shielded chain.  At $\kappa=1$, however, \cref{prop:kappa-one} shows that no accuracy-dependent lower bound is possible.
\item When $\min\{Q,\kappa\}>D_d^2$, the lower bound saturates at $\Omega(d^{4/3}/\log^{1/3}(ed))$.
\item The resisting frame is constructed against one deterministic transcript and does not directly yield a randomized lower bound.
\item The cap-avoidance argument uses $\norm{x_t}\le R$ and therefore does not cover unrestricted queries.  The hard objective is $C^{1,1}$, not necessarily $C^2$ or $C^\infty$.
\end{enumerate}

\section{Conclusion}

For deterministic adaptive exact-value optimization of globally smooth strongly convex functions under bounded queries, the currently certified landscape is
\[
\Omega\left(
 d\min\left\{
 \sqrt{\frac{\beta R^2}{\epsilon}},
 \sqrt\kappa,
 \left(\frac{d}{\log(ed)}\right)^{1/3}
 \right\}
\right)
\]
versus
\[
O\left(
 d\min\left\{
 \sqrt{\frac{\beta R^2}{\epsilon}},
 \sqrt\kappa\left[1+\logp{\frac{\mu R^2}{\epsilon}}\right]
 \right\}
\right).
\]
Consequently, the accuracy-dominated branch is exactly characterized up to constants, while the full high-accuracy region with
\[
\kappa\le\left(\frac{d}{\log(ed)}\right)^{2/3}
\]
has near-optimal $\widetilde\Theta(d\sqrt\kappa)$ exact-value complexity.  This conclusion holds for arbitrary high accuracy and is not confined to a constant-precision slice. The research pipeline was carried out almost entirely by
\ResearchAgentSystem{}, which also conducted a Lean-based review of the
manuscript.

\appendix

\section{Auxiliary convex-analytic facts}
\label{app:convex}

For completeness, we record two facts used in the main proof.

\begin{lemma}[Projection and squared distance]
Let $C\subseteq\R^d$ be nonempty, closed, and convex.  Then
\[
D_C(x):=\frac12\dist(x,C)^2
\]
is convex and continuously differentiable with
\[
\nabla D_C(x)=x-\Pi_Cx.
\]
Moreover, $\nabla D_C$ is $1$-Lipschitz.
\end{lemma}

\begin{proof}
The projection $\Pi_C$ is firmly nonexpansive:
\[
\norm{\Pi_Cx-\Pi_Cy}^2
\le\ip{\Pi_Cx-\Pi_Cy}{x-y}.
\]
Expanding the norm of $(I-\Pi_C)x-(I-\Pi_C)y$ shows that $I-\Pi_C$ is also firmly nonexpansive and hence nonexpansive.  Danskin's theorem, or a direct directional-derivative calculation using uniqueness of the Euclidean projection, gives the gradient formula.  Convexity follows from monotonicity of the gradient.
\end{proof}

\begin{lemma}[Chi-square lower tail]
If $g\sim N(0,I_p)$, then
\[
\Prob\left\{\norm g_2^2\le\frac p2\right\}\le e^{-p/16}.
\]
\end{lemma}

\begin{proof}
For any $t>0$, Markov's inequality gives
\[
\Prob\left\{e^{-t\norm g^2}\ge e^{-tp/2}\right\}
\le e^{tp/2}\E e^{-t\norm g^2}
=e^{tp/2}(1+2t)^{-p/2}.
\]
Taking $t=1/2$ yields
\[
\exp\left(\frac p4-\frac p2\log2\right)\le e^{-p/16}.
\]
\end{proof}

\section{A compact pseudocode description of the compiler}
\label{app:pseudocode}

\begin{algorithm}[H]
\DontPrintSemicolon
\KwIn{A fixed deterministic $T$-call algorithm; admissible $d,m,R,\beta_0$; block size $b=\lfloor d/4\rfloor$.}
\KwOut{A transcript, an orthonormal frame $U$, and one fixed hard objective $F_U$.}
Set $q=\lfloor T/b\rfloor$ and $s=T-qb$\;
\For{$j=1,\ldots,q$}{
  Answer the next $b$ adaptive queries using \eqref{eq:fullresponse}\;
  After the block is fixed, choose $u_j$ orthogonal to the block and cap-avoiding all earlier queries\;
}
Answer the final $s$ adaptive queries using \eqref{eq:lastresponse}\;
Compute the deterministic output $\widehat x$ from the completed transcript\;
Choose $u_{q+1}$ orthogonal to the incomplete block and $\widehat x$, and cap-avoiding all complete-block queries\;
Complete $u_{q+2},\ldots,u_m$ by cap avoidance on all queries and $\widehat x$\;
Define $F_U$ by \eqref{eq:FU}\;
\caption{Batched delayed-rotation fixed-objective lower-bound compiler.}
\end{algorithm}

\section{Constant bookkeeping}
\label{app:constants}

The lower proof keeps universal constants deliberately conservative.  One may take the compiler constant $c_0$ from \cref{cor:capparameter}; then choose
\[
a\le\min\left\{\frac14,c_0^{1/3},2^{-9}\right\}.
\]
After increasing $d_0$, the proof of \cref{thm:lower-main} is valid with
\[
c_\epsilon=\frac1{64},
\qquad
c=\frac a{16}.
\]
No effort is made to optimize these constants.  The upper proof uses $K=\lceil8\sqrt\kappa\rceil$ and admits a universal numerical constant in \eqref{eq:upper-main}.

\end{document}